\documentclass[11pt,a4paper]{article}
\usepackage{amsthm, amsmath, amssymb, titling, fullpage, color, bbm, mathrsfs, tikz}
\usepackage{fancyhdr,cancel,enumerate,array,booktabs,setspace,pgf,tikz,fancyhdr,graphicx,color}
\usepackage{float}
\usepackage{comment}

\usepackage{titling}
\usetikzlibrary{decorations.pathreplacing}
\usepackage{tikz-cd}
\newcommand{\Q}{{\mathbb{Q}}}
\newcommand{\R}{{\mathbb{R}}}
\newcommand{\Z}{{\mathbb{Z}}}
\newcommand{\C}{{\mathbb{C}}}

\newcommand{\PP}{{\mathbb{P}}}
\newcommand{\Proj}{\text{Proj}}
\newcommand{\Spec}{\text{Spec}}
\newcommand{\A}{\mathbb{A}}
\newcommand{\lcm}{\text{lcm}}
\newcommand{\OO}{\mathcal{O}}

\newtheorem*{theorem*}{Main Theorem}
\theoremstyle{plain}
\newtheorem{theorem}{Theorem}[section]
\theoremstyle{definition}
\newtheorem{definition}[theorem]{Definition}

\newtheorem{remark}{Remark}[subsection]
\newtheorem{prop}{Proposition}[section]
\newtheorem{lemma}[theorem]{Lemma}

\begin{document}
\date{}
\title{On the finiteness of loci of weighted plane curves in the moduli space}
\author{Monica Marinescu}
\maketitle
\begin{abstract}
For every fixed genus $g\geq 1$, we consider all quadruples $Q=(w_0,w_1,w_2,d)\in\Z^4_{>0}$ with the property that any smooth degree-$d$ curve embedded in the weighted projective plane $\PP^2(w_0,w_1,w_2)$ has genus $g$. We show there are infinitely many quadruples $Q$ satisfying this condition. For every such $Q$, we consider $Z_Q\subseteq M_g$ the locus in the moduli space of all smooth degree-$d$ curves embedded in $\PP^2(w_0,w_1,w_2)$. We show that, as $Q$ varies over all these quadruples, there are only finitely many different loci $Z_Q\subseteq M_g$.

\end{abstract}

\section{Introduction}

\hspace{6mm}Throughout this paper, we will work over $\mathbb{C}$. On $\PP^2$, a smooth curve of degree $d$ has genus $g=\frac{(d-1)(d-2)}{2}$. For a fixed genus $g\geq 1$, there exists at most one possible degree $d$ satisfying the formula. For such a pair $(g,d)$, let $N={{d+3}\choose{3}}-1$. The projective space $\PP^N$ naturally parametrizes all curves of degree $d$ embedded in $\PP^2$. Since smoothness is an open condition, there exists a (nonempty) open set $\mathcal{U}^{sm}\subseteq\PP^N$ parametrizing all smooth plane curves of degree $d$ and genus $g$. Consider the natural map $i:\mathcal{U}^{sm}\to M_g$, and let $Z=\overline{i(\mathcal{U}^{sm})}\subset M_g$ be the unique locus of plane curves in the moduli space $M_g$.

For smooth curves on weighted projective planes, the degree-genus formula no longer applies. In this situation, we have a formula for the genus in terms of the weights and the degree. Thus, fixing $g\geq 1$, we can look at all the quadruples $Q=(w_0,w_1,w_2,d)$ for which any smooth curve on $\PP^2(w_0,w_1,w_2)$ of degree $d$ has genus $g$. The set of smooth, degree-$d$ curves $C_d=V(f)\subset\PP^2(w_0,w_1,w_2)$  is nonempty if certain congruence relations between the weights and the degree are satisfied, as we will see in later sections. Lemma \ref{inf-quadruples} shows that for every $g\geq 1$, there are infinitely many quadruples $Q=(w_0,w_1,w_2,d)$ satisfying these congruence conditions. Any $Q$ with these properties will be called a $g$-good quadruple.

Since smoothness is an open condition, for any $g$-good quadruple $Q=(w_0,w_1,w_2,d)$, we get an open nonempty subset of smooth curves $\mathcal{U}^{sm}_Q\subset\PP V_Q$, where $V_Q$ is the vector space parametrizing all $\underline{w}$-weighted, degree-$d$ plane curves. Since every curve in $\mathcal{U}^{sm}_Q$ has genus $g$, there is a natural map $i:\mathcal{U}^{sm}_Q\to M_g$. Let $Z_Q=\overline{i(\mathcal{U}_Q)}$ be the corresponding irreducible closed locus in the moduli space. Our main statement is the following:

\begin{theorem*}
\normalfont
For any fix $g\geq 1$, there are only finitely many irreducible closed loci $Z_Q\subset M_g$, as $Q$ varies over all $g$-good quadruples.
\end{theorem*}

The strategy is as follows: first, we state the conditions needed for a general degree $d$-curve on $\PP^2(w_0,w_1,w_2)$ to be smooth, and give a genus formula for that case. We fix a genus $g\geq 1$, and let $Q=(w_0,w_1,w_2,d)$ vary over all $g$-good quadruples. For every such $Q$, we associate a polytope $P$ in the plane $H: xw_0+yw_1+zw_2=d\subset \R^3$, such that $P$ has exactly $g$ interior lattice points (i.e. points in $\Z^3$). We then use a linear transformation to obtain a polytope in $\Z^2$ with exactly $g$ interior lattice points. We prove there are only finitely many such polytopes in $\Z^2$, up to affine equivalence; from this, we conclude there are only finitely many irreducible closed loci $Z_Q\subset M_g$, as $Q$ varies over all $g$-good quadruples.

\section{Preliminaries}

\quad The results presented in this section can be found in  \cite{dolgachev} and \cite{fletcher}. The proofs, along with more detailed information, are contained in the source materials.

Let $W=\{w_0,\dots,w_n\}$ be a finite set of positive integers and $|W|=w_0+\dots+w_n$. Let $\C(W)$ be the graded polynomial algebra $\C[x_0,\dots,x_n]$ over the field $\C$, with grading given by $\deg(x_i)=w_i$, for $i=0,\dots, n$. We say  $\PP^n(W)= \Proj(\C(W))$ is a weighted projective space of type $W$.

Let $U=\A^{n+1}-\{0\}=\Spec(\C(W))-\{\mathfrak{m}\}$,  where $\mathfrak{m}=(x_0,\dots,x_n)$. The grading on $\C(W)$ is equivalent to an action of the algebraic torus $G_m=\Proj(\C[t,t^{-1}])$ on $\Spec(\C(W))$, which on points is given by: 
\begin{align*}
\C^*\times \C^{n+1}& \to \C^{n+1}\\
(t,(a_0,\dots,a_n))& \to(a_0t^{w_0},\dots,a_nt^{w_n}).
\end{align*}

\begin{lemma}\cite[1.2.1]{dolgachev}
The open set $U=\A^{n+1}-\{0\}$ is invariant with respect to this action and the universal geometric quotient $U/G_m$ exists and coincides with $\PP^n(W)$.
\end{lemma}

\begin{definition}\cite[5.11]{fletcher}
The weighted projective space $\PP^n(W)$ is called {\it well-formed} if, for every $i=0,\dots,n$, we have
$\gcd(w_0,\dots,\hat{w_i},\dots,w_n)=1.$
\end{definition}

\begin{lemma}\cite[1.3.1]{dolgachev}
For any finite set $W$, let $d_i=\gcd(w_0,\dots,\hat{w_i},\dots,w_n)$, and $a_i=\lcm(d_0,\dots,\hat{d_i},\dots,d_n)$. Then $\PP(W)\cong \PP(W')$, where $W'=\{\frac{w_0}{a_0},\dots,\frac{w_n}{a_n}\}$.
Moreover, for any $i=0,\dots,n$, we have
$$\gcd(w'_0,\dots,\hat{w_i}',\dots,w'_n)=1.$$
\end{lemma}

\begin{remark}
As a consequence of the lemma above, we can always assume the weighted projective space $\PP^n(W)$ is well-formed.
\end{remark}

All weighted projective spaces $\PP^n(W)$ are normal irreducible projective algebraic varieties. We point out a few pathologies that make them quite different from the usual projective space $\PP^n=\PP(1,\dots,1)$:

\begin{enumerate}
\item[(i)] $\PP^n(W)$ has cyclic quotient singularities; $\PP^n(W)$ is nonsingular if and only if $W$ reduces to $(1,\dots,1)$ and $\PP^n(W)\cong\PP^n$.
\item[(i)] On $\PP^n$, for any $k\in\Z$, $\OO_{\PP^n}(k)$ is an invertible sheaf. This is no longer true on weighted projective spaces.
\item[(ii)] On $\PP^n$, for any $k,l\in\Z$, we have an isomorphism $\OO_{\PP^n}(k)\otimes\OO_{\PP^n}(l)\cong\OO_{\PP^n}(k+l)$. This is no longer true on weighted projective spaces.
\end{enumerate}

\begin{definition}\cite[3.1.1]{dolgachev}
Let $X\hookrightarrow \PP^n(W)$ be a closed subvariety. Let $U=\Spec(\C(W))-\{\mathfrak{m}\}\subset \A^{n+1}$, where $\mathfrak{m}=(x_0,\dots, x_n)$, and let $p:U\to \PP^n(W)$ the canonical projection. The scheme closure of $C_X^*=p^{-1}(X)$ in $\A^{n+1}$, denoted $C_X$, is called {\it the affine quasicone} over $X$. The variety $X$ is called {\it quasismooth} with respect to the embedding in $\PP^n(W)$ if the affine quasicone is smooth outside its vertex.
\end{definition}

Now we turn our focus to smooth weighted plane curves. The first results can be found in \cite{fletcher}. For more information, please see \cite{fletcher}, Sections 8 and 12.

 \begin{theorem}\normalfont \cite[12.1]{fletcher}
 A weighted curve complete intersection is smooth if and only if it is quasismooth.
 \end{theorem}
 
 \begin{theorem}\normalfont
 \cite[8.4]{fletcher}
 \label{monomial-condition}
Let $C_d=V(f)\subset\PP^2(W)$ be a general curve of degreee $d$, where $d>w_i,\forall i$. The curve is (quasi)smooth if and only if the following conditions hold for all $i$:
 \begin{enumerate}
 \item[(i)] $f$ contains a monomial $x_i^kx_j$ of degree $d$ for some power $k$ and index $j$, not necessarily different than $i$.
 \item[(ii)] $f$ contains a monomial of degree $d$ which does not involve the variable $x_i$.
 \end{enumerate}
 \end{theorem}

 \begin{theorem}
 \normalfont
 \cite[12.2]{fletcher}
  Let $C_d$ be a smooth curve on $\PP^2(W)$. The genus $g$ of $C_d$ is given by the following formula:
 \begin{equation}\label{genus-formula}
 g=\frac{1}{2}(\frac{d(d-\sum_i w_i)}{w_0w_1w_2}+\sum_{i=0}^2\frac{\gcd(w_i,d)}{w_i}-1).
 \end{equation}
 \end{theorem}
 
 \begin{remark}
 In Theorem \ref{monomial-condition}, the degree $d$ has to be strictly higher than the weights. Otherwise, say $d=w_0$. Then $f(x_0,x_1,x_2)=x_0+g(x_1,x_2)$, for some $g\in\C[x_1,x_2]$. This implies $C_d\cong \PP^1(w_1,w_2)$, which is either not smooth, or is smooth of genus 0.
 \end{remark}

  \begin{definition}
 \normalfont
 \label{good-quadruple}
 We call $Q=(w_0,w_1,w_2,d)$ a {\it good quadruple} if $w_0,w_1,w_2\in\Z_{\geq 1}$ are pairwise coprime, $d>w_i,\forall i$, and the congruence conditions in Theorem \ref{monomial-condition} can be satisfied. We call $Q=(w_0,w_1,w_2,d)$ a {\it g-good quadruple} if it is a good quadruple, and every smooth curve on $\PP(w_0,w_1,w_2)$ of degree $d$ has genus $g$.
 \end{definition}
 

\begin{prop}
\label{inf-quadruples}
 For any $g\geq 1$, there are infinitely many $g$-good quadruples $Q=(w_0,w_1,w_2,d)$. 
 \end{prop}
 
\begin{proof}
Fix a genus $g\geq 1$. Let $d\geq 1$ such that $d\equiv -1\mod (2g+2)$. We claim $(1,\frac{d-1}{2},\frac{d+1}{2g+2},d)$ is a $g$-good quadruple. The weights $1,\frac{d-1}{2},\frac{d+1}{2g+2}$, are pairwise coprime, and the congruence conditions of Theorem \ref{monomial-condition} are satisfied by the monomials $x_0^d, x_1^2x_0, x_2^{g+1}x_1$. Applying the genus formula (\ref{genus-formula}) for a general curve $C_d$, we obtain $g(C_d)=g$, as desired.
\end{proof}

Fix $g\geq 1$. For any $Q=(w_0,w_1,w_2,d)$ a $g$-good quadruple, let $\mathcal{U}^{sm}_Q\subset \PP V_Q$ be the (nonempty) open set of smooth curves of degree $d$ on $\PP(w_0,w_1,w_2)$, where $V_Q$ is the vector space parametrizing degree-$d$ curves on $\PP^2(w_0,w_1,w_2)$. Let $i:\mathcal{U}^{sm}_Q\to M_g$ be the natural map to the moduli space, and let $Z_Q=\overline{i(\mathcal{U}^{sm}_Q)}$ be the corresponding closed irreducible subset in $M_g$. Then the main theorem of this paper states that there are only finitely many loci $Z_Q\in M_g$, as $Q$ varies over all $g$-good quadruples.

\section{Proof of Main Theorem}

\hspace{6mm}Polytopes arise naturally in this setting. For a fixed genus $g\geq 1$, let $(w_0,w_1,w_2, d)$ be a $g$-good quadruple. To every monomial $x_0^ax_1^bx_2^c$ of degree $d$, we associate the point $(a,b,c)\in\Z^3$. These points represent all the lattice points in the plane H: $xw_0+yw_1+zw_2=d$, situated in the first octant $(\R_{\geq 0})^3$. Let $P$ be the convex hull of these points.

We denote by $n(P)$ the number of lattice points in $P$. To any polytope $P$ we associate the $n(P)\times 3$ matrix $M(P)$ whose rows represent the coordinates of the lattice points.\\

Because $(w_0,w_1,w_2,d)$ is a $g$-good quadruple, the polytope $P$ will contain the following points [Thm. \ref{monomial-condition}]:
\begin{enumerate}
\item[(a)] for each coordinate $x, y, z$, there exists a point in $P$ whose other two coordinates have sum at most 1, i.e.: we either have $(a,0,0)$ or $(a,1,0)$ or $(a,0,1)$; we either have $(0,b,0)$ or $(1,b,0)$ or $(0,b,1)$; we have $(0,0,c)$ or $(1,0,c)$ or $(0,1,c)$.

\item[(b)] for each coordinate $x,y,z$, there exists a point on $P$ which has that coordinate zero, i.e. $P$ has points of the form $(0,b,c),(a',0,c'),(a'',b'',0)$.\\
\end{enumerate}
 
 We call the vectors in (a) {\it distinguished}. In most cases, we have just three distinguished vectors, but in special circumstances we can have more. For example, $(a,0,0)$ and $(a',1,0)$ can both be in $P$; in this case, we focus on $(a,0,0)$. If $(a,0,1)$ and $(a',1,0)$ are in $P$, then we focus on either of them. In addition, we observe that these three vectors are all different, since $a,b,c> 1$; otherwise, $P$ wouldn't have any interior points, which is not possible by Lemma \ref{g-int-pts}. The triangle with the three picked vertices will be called {\it the distinguished triangle $\Delta$}. The corresponding {\it distinguished} $3\times 3$ minor in $M(P)$ will also be named $\Delta$, by abuse of notation.
  
 \begin{remark}
 Let $(w_0,w_1,w_2,d)$ be a good quadruple. Then for every $i$, we have either $w_i|d$, or $\gcd(w_i,d)=1$. To see this, assume by contradiction that $\exists i$ such that $1<\gcd(w_i,d)<w_i$. Say $i=0$; then the associated polytope $P$ does not have a point of the form $(a,0,0)$, so it must have either $(a,1,0)$ or $(a,0,1)$. If $\gcd(w_0,d)=l>1$, then $l$ divides either $w_1$ or $w_2$, as well, contradicting the condition of the weights being coprime.
 \end{remark}
 
 \begin{lemma}
 \label{g-int-pts}
 Let $(w_0,w_1,w_2,d)$ be a $g$-good quadruple, and let $P$ be its associated polytope. Then the number of interior lattice points of $P$ equals $g$.
 \end{lemma}
 
 \begin{proof}
 Let $Q=(w_0,w_1,w_2,d)$ be a $g$-good quadruple. Our first claim is that $g$ equals the number of nonnegative triples $(a,b,c)\in\Z^3_{\geq 0}$ such that $aw_0+bw_1+cw_2=d-w_0-w_1-w_2$; equivalently, $g$ equals the number of lattice points in the polytope $P$ that are not on the boundary of the octant.
 
 To see this, let $C_d=V(f)$ be a general curve of degree $d$ on $\PP^2(W)$ (thus also smooth of genus $g$). Let $J_f=(\frac{\partial f}{\partial x_0},\frac{\partial f}{\partial x_1}, \frac{\partial f}{\partial x_2})$ be the Jacobian ideal corresponding to polynomial $f$. A theorem by Steenbrink (\cite[4.3.2]{dolgachev}) states the following about the genus $g$:
\begin{equation}
g=h^{0,1}(C_d)=h^{1,0}(C_d)=\dim_{\C}(\C(W)/J_f)_{d-|W|}.
\end{equation}

Since the generators of $J_f$ have degrees $d-w_i$, for $i=0,1,2$, which are strictly bigger than $d-|W|$, it follows that 
\begin{equation}
g=\dim_{\C}(\C(W))_{d-|W|}=\#\{(a,b,c)\in\Z_{\geq 0}|aw_0+bw_1+cw_2=d-|W|\}.
\end{equation}
 
The second step is to show that any lattice point in $P$ which is not on the boundary of the octant is actually an interior point of $P$.

\begin{figure}[H]
\centering
\begin{tikzpicture}[x=0.5cm,y=0.5cm,z=0.3cm,>=stealth]
\node[circle,fill=black,inner sep=0.8pt,draw] (a) at (7,0) {};
\node[circle,fill=black,inner sep=0.8pt,draw] (b) at (0,7) {};
\node[circle,fill=black,inner sep=0.8pt,draw] (c) at (-4,-4) {};
\node[circle,fill=black,inner sep=0.8pt,draw] (d) at (-2.1666,-3.333) {};
\node[circle,fill=black,inner sep=0.8pt,draw] (e) at (-2.666,-0.333) {};
\node[circle,fill=black,inner sep=0.8pt,draw] (f) at (-0.666, 5.1666) {};
\node[circle,fill=black,inner sep=0.8pt,draw] (g) at (1.1666, 5.8333) {};

\draw[dashed][->] (xyz cs:x=0) -- (xyz cs:x=8.5) node[above] {$y$};
\draw[dashed][->] (xyz cs:y=0) -- (xyz cs:y=8.5) node[right] {$z$};
\draw[dashed][->] (xyz cs:z=0) -- (xyz cs:z=-8.5) node[above] {$x$};
\draw(a)--(8.5,0);
\draw(b)--(0,8.5);
\draw(c)--(-5,-5);
\draw(a)--(b)--(c)--(a);
\draw[line width=0.6mm](d)--(e)--(f)--(g)--(a)--(d);

\node[circle,fill=black,inner sep=0.8pt,draw] (h) at (15,0) {};
\node[circle,fill=black,inner sep=0.8pt,draw] (i) at (24,0) {};
\node[circle,fill=black,inner sep=0.8pt,draw] (j) at (22,7) {};
\node[circle,fill=black,inner sep=0.8pt,draw] (k) at (16.8,0) {};
\node[circle,fill=black,inner sep=0.8pt,draw] (l) at (18.5,3.5) {};
\node[circle,fill=black,inner sep=0.8pt,draw] (m) at (19.5,3.5) {};
\node[circle,fill=black,inner sep=0.8pt,draw] (n) at (23,2.5) {};
\node[circle,fill=black,inner sep=0.8pt,draw] (o) at (21.5,1) {};

\draw(h)--(24,0);
\draw(h)--(22,7);
\draw[line width=0.6mm](i)--(k)--(l)--(j);
\draw[dashed][line width=0.6mm](i)--(26,0);
\draw[dashed][line width=0.6mm](j)--(24,9);

\node () at (13.5,-0.5) {\tiny$(\frac{d}{w_0},0,0)$};
\node () at (16.8,-1) {\tiny$A=(a,1,0)$};
\node () at (17.1,3.78) {\tiny$B=(c,0,d)$};

\end{tikzpicture}
\caption{Polytope in the case $w_0, w_2\nmid d, w_1|d$}
\label{3d-polytope}
\end{figure}
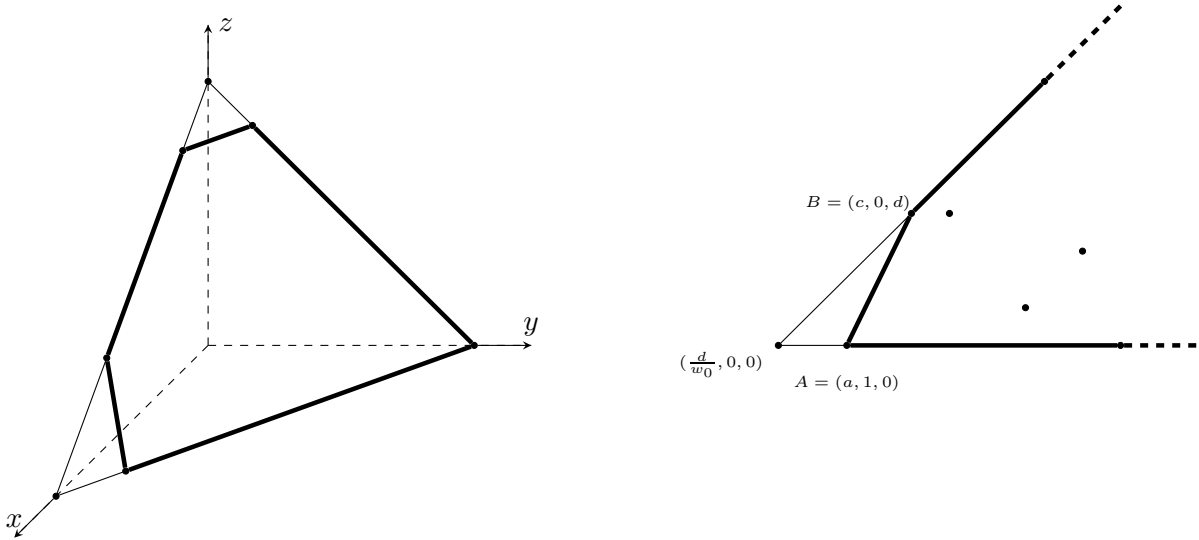

 Our polytope looks similar to the one in Figure \ref{3d-polytope}. We only need to check what happens around the axes corresponding to the variables whose weights don't divide $d$. Say $w_0\nmid d$, therefore $(a,1,0)$ or $(a,0,1)$ are in $P$; suppose $A=(a,1,0)\in P$. Pick the closest lattice point $B=(c,0,d)$ on the other side (we know such a point exists). Any point on the open segment $(AB)$ or between the segment and the vertex $(\frac{d}{w_0},0,0)$ has $y$-coordinate smaller than 1, which means it cannot be a lattice point. Therefore, there are no other lattice points of $P$ in this region. We conclude that a lattice point $(a,b,c)$ is in the interior of $P$ if and only if $a,b,c\geq 1$, thus the number of interior lattice points is exactly $g$.
 \end{proof}

 \begin{lemma}
 \label{all-minors-div-d}
Let $(w_0,w_1,w_2,d)$ be a $g$-good quadruple. Then all $3\times 3$ minors of $M(P)$ have determinant divisible by $d$.
 \end{lemma}
 
 \begin{proof}
Let 
 $$A=\begin{pmatrix} 
 x_0&&y_0&&z_0\\
 x_1&&y_1&&z_1\\
 x_2&&y_2&&z_2\end{pmatrix}$$
 be a minor of $M(P)$.
 We have the following:
 
\begin{align*}
 w_0\det(A) =& \det\begin{pmatrix}
 w_0x_0&&y_0&&z_0\\
 w_0x_1&&y_1&&z_1\\
 w_0x_2&&y_2&&z_2\end{pmatrix}
= \det\begin{pmatrix}
 w_0x_0+w_1y_0&&y_0&&z_0\\
 w_0x_1+w_1y_1&&y_1&&z_1\\
 w_0x_2+w_1y_2&&y_2&&z_2\end{pmatrix}=\\
 =&\det\begin{pmatrix}
 w_0x_0+w_1y_0+w_2z_0&&y_0&&z_0\\
 w_0x_1+w_1y_1+w_2z_1&&y_1&&z_1\\
 w_0x_2+w_1y_2+w_2z_2&&y_2&&z_2\end{pmatrix}=d\cdot\det\begin{pmatrix}
 1&&y_0&&z_0\\
 1&&y_1&&z_1\\
 1&&y_2&&z_2\end{pmatrix}\in d\Z
\end{align*}
Thus $\det(A)\in\frac{d}{w_0}\Z$. A similar argument using $w_1$ instead of $w_0$ shows $\det(A)\in\frac{d}{w_1}\Z$. Since $\gcd(w_0,w_1)=1$, it follows that $\det(A)\in d\Z$.
 \end{proof}
 
 \begin{theorem}
 \normalfont
 \label{minor-d}
Let $Q=(w_0,w_1,w_2,d)$ be a good $g$-quadruple, with its associated polytope $P$ and $n(P)\times 3$ matrix $M(P)$. Let $\Delta$ be the distinguished triangle corresponding to $Q$, as defined in the beginning of the section. Let $v_1,v_2,v_3$ be any three lattice points in $\Delta(P)$ which form a triangle that doesn't contain any other lattice points, either on the boundary or in the interior. Then the corresponding $3\times 3$ minor $\begin{pmatrix}
-&v_1&-\\
-&v_2&-\\
-&v_3&-\end{pmatrix}$ of $M(P)$ has determinant $\pm d$.
 \end{theorem}
Before we prove Theorem \ref{minor-d}, we need an additional lemma:
 
\begin{lemma}\label{triangulation}
A triangle which contains $m$ lattice points in the interior and $n$ lattice points on the boundary (including the three vertices), where $n\geq 3$, can be divided into $2m+n-2$ smaller triangles with disjoint interiors, whose vertices are among the $m+n$ lattice points, such that any small triangle contains no other lattice points other than its vertices.
 \end{lemma}
 
 \begin{proof}[Proof of lemma]
 We start with the boundary points and proceed by induction. If $n=3$ and $m=0$, we have just one triangle. For $n>3$, pick a lattice point in the interior of an edge and connect it to the opposite vertex, obtaining two triangles with disjoint interiors, say one with $l$ boundary lattice points, one with $n-l+2$ boundary lattice points. By induction, we can further divide these into $(l-2)+(n-l)=n-2$ triangles.\\
 
\begin{figure}[H]
\centering
\begin{tikzpicture}
\node[circle,fill=black,inner sep=1.2pt,draw] (a) at (0,0) {};
\node[circle,fill=black,inner sep=1.2pt,draw] (b) at (2,0) {};
\node[circle,fill=black,inner sep=1.2pt,draw] (c) at (4,0) {};
\node[circle,fill=black,inner sep=1.2pt,draw] (d) at (1,2) {};
\node[circle,fill=black,inner sep=1.2pt,draw] (e) at (.5,1) {};
\node[circle,fill=black,inner sep=1.2pt,draw] (f) at (3.4,.4) {};
\node[circle,fill=black,inner sep=1.2pt,draw] (g) at (2.8,.8) {};
\node[circle,fill=black,inner sep=1.2pt,draw] (h) at (1.6,1.6) {};

\draw (a)--(b)--(c)--(f) --(g)--(h)--(d)--(e)--(a);
\draw (e)--(b)--(f);
\draw (b)--(g);
\draw (b)--(h);
\draw (b)--(d);
\end{tikzpicture}
\caption{Case m=0, n=8}
\end{figure}
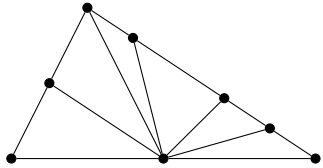

For the interior lattice points we have a similar stategy, starting from the above triangulation. For each interior lattice point, we are in one of two cases: either the point is in the interior of one of the smaller triangles, or it falls on the boundary of exactly two smaller triangles.

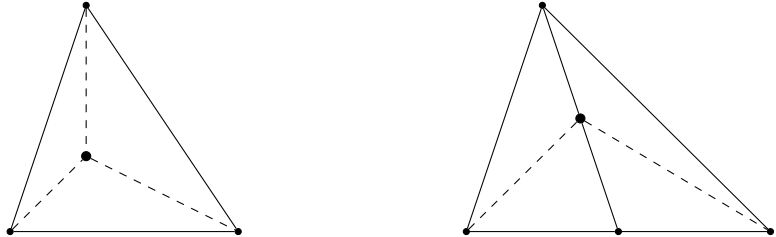
\begin{figure}[H]
\centering
\begin{tikzpicture}
\node[circle,fill=black,inner sep=0.8pt,draw] (a) at (0,0) {};
\node[circle,fill=black,inner sep=0.8pt,draw] (b) at (3,0) {};
\node[circle,fill=black,inner sep=0.8pt,draw] (c) at (1,3) {};
\node[circle,fill=black,inner sep=1.2pt,draw] (d) at (1,1) {};

\node[circle,fill=black,inner sep=0.8pt,draw] (e) at (6,0) {};
\node[circle,fill=black,inner sep=0.8pt,draw] (f) at (10,0) {};
\node[circle,fill=black,inner sep=0.8pt,draw] (g) at (7,3) {};
\node[circle,fill=black,inner sep=0.8pt,draw] (h) at (8,0) {};
\node[circle,fill=black,inner sep=1.2pt,draw] (i) at (7.5,1.5) {};

\draw (a)--(b)--(c)--(a);
\draw[dashed] (a)--(d)--(b);
\draw[dashed](c)--(d);

\draw (e)--(h)--(f)--(g)--(e);
\draw (g)--(i)--(h);
\draw[dashed](e)--(i)--(f);

\end{tikzpicture}
\caption{Two possible cases for m=1}
\end{figure}

In the first case, we connect that interior lattice point with the three vertices of the triangle, from one triangle obtaining three. In the second case, we connect the interior lattice point with the opposite vertices, from two triangles obtaining four. By induction, after we add $m$ interior points, we get $2m$ more triangles, so the conclusion follows.
 \end{proof}

 \begin{proof}[Proof of Theorem \ref{minor-d}]
 Consider $\Delta$ the distinguished triangle as described in the beginning of this section. We will make a case-by-case analysis of the number of boundary lattice points and interior lattice points of $\Delta$. We have the following cases:
 \begin{enumerate}
 
\item[(a)] None of the weights divide $d$. Up to a permutation of the weights, the distinguished triangle corresponds to one of these matrices:
$$(i)\begin{pmatrix}
a&&1&&0\\
0&&b&&1\\
1&&0&&c\end{pmatrix}\text{ or }
(ii)\begin{pmatrix}
a&&1&&0\\
0&&b&&1\\
0&&1&&c\end{pmatrix}.$$

\item[(i)] Say the distinguished triangle is $\begin{pmatrix}
a&&1&&0\\
0&&b&&1\\
1&&0&&c\end{pmatrix}.$ Then $\det(\Delta)=abc+1$, and $a=\frac{d-w_1}{w_0}, b=\frac{d-w_2}{w_1}, c=\frac{d-w_0}{w_2}.$

$$
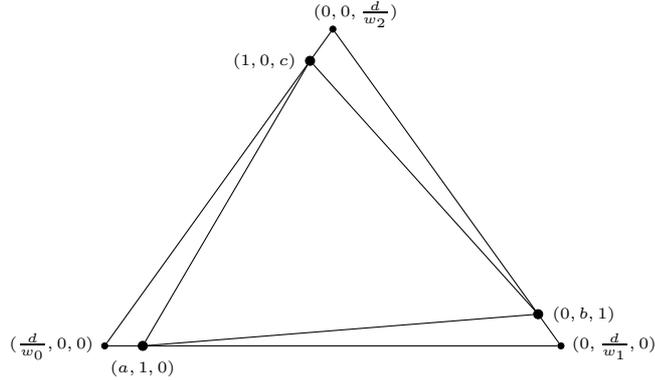
\begin{tikzpicture}
\node[circle,fill=black,inner sep=0.8pt,draw] (a) at (0,0) {};
\node[circle,fill=black,inner sep=0.8pt,draw] (b) at (6,0) {};
\node[circle,fill=black,inner sep=0.8pt,draw] (c) at (3,4.2) {};
\node[circle,fill=black,inner sep=1.2pt,draw] (d) at (.5,0) {};
\node[circle,fill=black,inner sep=1.2pt,draw] (e) at (5.7,.42) {};
\node[circle,fill=black,inner sep=1.2pt,draw] (f) at (2.7,3.78) {};

\node () at (-0.7,0) {\tiny$(\frac{d}{w_0},0,0)$};
\node () at (6.7,0) {\tiny$(0,\frac{d}{w_1},0)$};
\node () at (3.3,4.4) {\tiny$(0,0,\frac{d}{w_2})$};
\node () at (0.5,-0.3) {\tiny$(a,1,0)$};
\node () at (6.3,.42) {\tiny$(0,b,1)$};
\node () at (2.1,3.78) {\tiny$(1,0,c)$};

\draw (a)--(b)--(c)--(a);
\draw (d)--(e)--(f)--(d);
\end{tikzpicture}$$

The triangle $\Delta$ has no lattice points on the boundary other than the three vertices, but it has $g$ interior lattice points, by Lemma \ref{g-int-pts}. Triangulating using Lemma \ref{triangulation}, we get $2g+1$ smaller triangles with disjoint interiors, whose union is $\Delta$. All these triangles form, together with the origin, small tetrahedrons of volume equal to positive multiples of $\frac{d}{2}$, by Lemma \ref{all-minors-div-d}. It follows that $\Delta$ together with the origin forms a tetrahedron of volume at least $\frac{(2g+1)d}{2}$, i.e. $\det(\Delta)\geq (2g+1)d$.

Assume by contradiction that one of the small tetrahedrons has volume at least $d$. Then we must have that:
\begin{align*}
\det(\Delta)=abc+1 &> d(2g+1)\\
\frac{(d-w_0)(d-w_1)(d-w_2)}{w_0w_1w_2}+1 &> d(2g+1)\\
d(d-\sum w_i)+\sum w_iw_j &> w_0w_1w_2(2g+1)
\end{align*}
On the other hand, the genus formula (\ref{genus-formula}) gives:
$$d(d-\sum w_i)+\sum w_iw_j =w_0w_1w_2(2g+1),$$

so we reached a contradiction. All the small tetrahedrons must have volume exactly $\frac{d}{2}$, i.e. the determinants corresponding the associated $3\times 3$ minors of $M(P)$ equal $\pm d$.

\item[(ii)] Say the distinguished triangle is $\begin{pmatrix}
a&&1&&0\\
0&&b&&1\\
0&&1&&c\end{pmatrix}.$
Then $\det(\Delta)=a(bc-1)$. Since the weights are pairwise comprime, there exist $k,l\in\Z_{>0}$ such that
\begin{align*}
a=&lw_2\\
b=&1+kw_2\\
c=&1+kw_1=lw_0\\
d=&kw_1w_2+w_1+w_2=lw_0w_2+w_1.
\end{align*}
The genus formula (\ref{genus-formula}) gives 
\begin{align*}
(2g+1)w_0w_1w_2&=d(d-\sum w_i)+\sum w_iw_j\\
(2g+1)w_0w_1w_2&=d(kw_1w_2-w_0)+\sum w_iw_j\\
(2g+1)w_0w_1w_2&=dkw_1w_2-w_0(d-w_1-w_2)+w_1w_2\\
(2g+1)w_0w_1w_2&=dkw_1w_2-kw_0w_1w_2+w_1w_2\\
(2g+1)w_0&=dk-kw_0+1\\
(2g+1)w_0&=(aw_0+w_1)k-kw_0+1\\
(2g+1)w_0&=aw_0+(w_1k+1)-w_0k\\
(2g+1)w_0&=aw_0+lw_0-w_0k\\
2g+1&=ak+l-k.
\end{align*}
Hence we have
$$\det(\Delta)=a(bc-1)=akd=(2g+1+k-l)d.
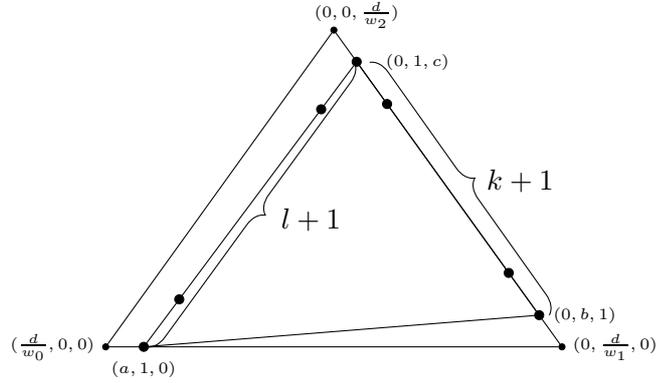
$$

$$\begin{tikzpicture}[level/.style={},decoration={brace,mirror,amplitude=7}][H]
\node[circle,fill=black,inner sep=0.8pt,draw] (a) at (0,0) {};
\node[circle,fill=black,inner sep=0.8pt,draw] (b) at (6,0) {};
\node[circle,fill=black,inner sep=0.8pt,draw] (c) at (3,4.2) {};
\node[circle,fill=black,inner sep=1.2pt,draw] (d) at (.5,0) {};
\node[circle,fill=black,inner sep=1.2pt,draw] (e) at (5.7,.42) {};
\node[circle,fill=black,inner sep=1.2pt,draw] (f) at (3.3,3.78) {};
\node[circle,fill=black,inner sep=1.2pt,draw] (h) at (2.8333,3.15) {};
\node[circle,fill=black,inner sep=1.2pt,draw] (i) at (0.9666,0.63) {};
\node[circle,fill=black,inner sep=1.2pt,draw] (j) at (5.3, 0.98) {};
\node[circle,fill=black,inner sep=1.2pt,draw] (k) at (3.7, 3.22) {};

\node () at (-0.7,0) {\tiny$(\frac{d}{w_0},0,0)$};
\node () at (6.7,0) {\tiny$(0,\frac{d}{w_1},0)$};
\node () at (3.3,4.4) {\tiny$(0,0,\frac{d}{w_2})$};
\node () at (0.5,-0.3) {\tiny$(a,1,0)$};
\node () at (6.3,.42) {\tiny$(0,b,1)$};
\node () at (4.1,3.78) {\tiny$(0,1,c)$};

\draw (a)--(b)--(c)--(a);
\draw (d)--(e)--(f)--(d);
\draw [decorate] ([xshift=2mm]e.west) --node[xshift=8mm][yshift=1mm]{$k+1$} ([xshift=1mm]f.east);
\draw [decorate] ([xshift=1mm]d.west) --node[xshift=8mm][yshift=-2mm]{$l+1$} ([xshift=-1mm]f.east);
\end{tikzpicture}$$

The distinguished triangle $\Delta$ has $k+l+1$ lattice points on the boundary, and $g-l+1$ lattice points in the interior, so we can triangulate $\Delta$ into $2(g-l+1)+k+l-1=2g+1+k-l$ smaller triangles.\\
Assume by contradiction that at least one small tetrahedron has volume $\geq d$, then 
$$\det(\Delta)>d(2g+1+k-l),$$
which is clearly a contradiction.

 \item[(b)] Only one weight divides $d$. Up to a permutation of the weights, the three possible cases are the following:
 $$(i)\begin{pmatrix}
 a&&0&&0\\
 0&&b&&1\\
 0&&1&&c\end{pmatrix}\text{ or }
 (ii)\begin{pmatrix}
 a&&0&&0\\
 1&&b&&0\\
 0&&1&&c\end{pmatrix}\text{ or }
(iii) \begin{pmatrix}
 a&&0&&0\\
 1&&b&&0\\
 1&&0&&c\end{pmatrix}.$$

 \item[(i)] If the distinguished triangle is $\begin{pmatrix}
 a&&0&&0\\
 0&&b&&1\\
 0&&1&&c\end{pmatrix},$
there exists $k\in\Z$ such that:
\begin{align*}
&b=kw_2+1\\
&c=kw_1+1\\
&d=aw_0=kw_1w_2+w_1+w_2\\
&\det(\Delta)=akd
\end{align*}
 The genus formula (\ref{genus-formula}) gives:
 $$g=\frac{1}{2}(\frac{d(d-w_0-w_1-w_2)}{w_0w_1w_2}+1+\frac{1}{w_1}+\frac{1}{w_2}-1).$$
 Rewriting this equality and using the information above, we get:
 \begin{align*}
 2gw_0w_1w_2=&d(d-\sum_i w_i)+w_0w_1+w_0w_2\\
2gw_1w_2=&a(kw_1w_2-w_0)+w_1+w_2\\
2gw_1w_2=&akw_1w_2-d+w_1+w_2\\
2gw_1w_2=&akw_1w_2-kw_1w_2\\
2g=&k(a-1)\\
a=&\frac{2g+k}{k}.
\end{align*}
 
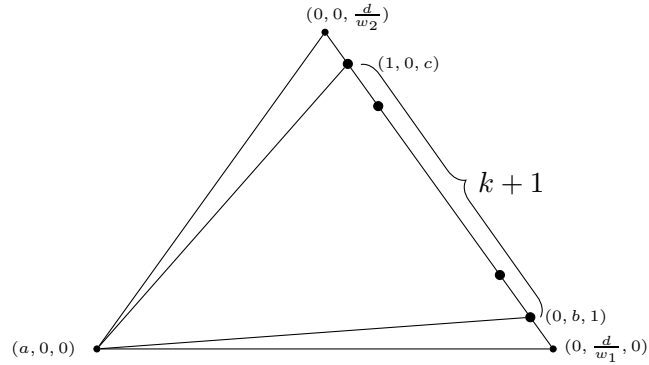
$$\begin{tikzpicture}[level/.style={},decoration={brace,mirror,amplitude=7}][H]
\node[circle,fill=black,inner sep=0.8pt,draw] (a) at (0,0) {};
\node[circle,fill=black,inner sep=0.8pt,draw] (b) at (6,0) {};
\node[circle,fill=black,inner sep=0.8pt,draw] (c) at (3,4.2) {};
\node[circle,fill=black,inner sep=1.2pt,draw] (e) at (5.7,.42) {};
\node[circle,fill=black,inner sep=1.2pt,draw] (f) at (3.3,3.78) {};
\node[circle,fill=black,inner sep=1.2pt,draw] (j) at (5.3, 0.98) {};
\node[circle,fill=black,inner sep=1.2pt,draw] (k) at (3.7, 3.22) {};

\node () at (-0.7,0) {\tiny$(a,0,0)$};
\node () at (6.7,0) {\tiny$(0,\frac{d}{w_1},0)$};
\node () at (3.3,4.4) {\tiny$(0,0,\frac{d}{w_2})$};
\node () at (6.3,.42) {\tiny$(0,b,1)$};
\node () at (4.1,3.78) {\tiny$(1,0,c)$};

\draw (a)--(b)--(c)--(a);
\draw (e)--(a)--(f);
\draw [decorate] ([xshift=2mm]e.west) --node[xshift=8mm][yshift=1mm]{$k+1$} ([xshift=1mm]f.east);
\end{tikzpicture}$$
 
 The triangle $\Delta$ has $g$ interior lattice points and $k+2$ boundary lattice points, so by Lemma \ref{triangulation}, we can partition it into $2g+k$ triangles. Assume by contradiction one small tetrahedron has volume at least $d$, so $\det(\Delta)>(2g+k)d$. On the other hand,
 $$\det(\Delta)=akd=\frac{2g+k}{k}kd=(2g+k)d,$$
 so we reached a contradiction again.
 
 \item[(ii)]
 If the distinguished triangle is $\begin{pmatrix}
a&&0&&0\\
1&&b&&0\\
0&&1&&c\end{pmatrix},$
 there exists $k\in\Z$ such that:
 \begin{align*}
  &a=1+kw_1\\
  &b=kw_0\\
  &d=kw_0w_1+w_0=cw_2+w_1\\
  &\det(\Delta)=abc=kdc
  \end{align*}
 Again, we know the genus is given by (\ref{genus-formula}):
 $$g=\frac{1}{2}(\frac{d(d-w_0-w_1-w_2)}{w_0w_1w_2}+1+\frac{1}{w_1}+\frac{1}{w_2}-1).$$
 Rewriting this formula and using the information above, we obtain:

\begin{align*}
2gw_0w_1w_2=&d(d-\sum_i w_i)+w_0w_1+w_0w_2\\
2gw_0w_1w_2=&aw_0(kw_0w_1+w_0-\sum_i w_i)+w_0w_1+w_0w_2\\
2gw_1w_2=&a(kw_0w_1-w_1-w_2)+w_1+w_2\\
2gw_1w_2=&(kw_1+1)(kw_0w_1-w_1-w_2)+w_1+w_2\\
2gw_1w_2=&kw_1(kw_0w_1-w_1-w_2)+kw_0w_1\\
2gw_2=&k(kw_0w_1+w_0-w_1-w_2)\\
2gw_2=&k(d-w_1-w_2)\\
2gw_2=&k(c-1)w_2\\
2g=&k(c-1)\\
c=&\frac{2g+k}{k}.
\end{align*}

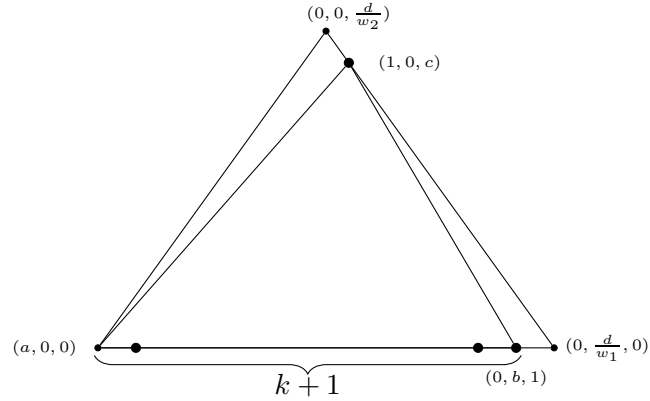
$$\begin{tikzpicture}[level/.style={},decoration={brace,mirror,amplitude=7}][H]
\node[circle,fill=black,inner sep=0.8pt,draw] (a) at (0,0) {};
\node[circle,fill=black,inner sep=0.8pt,draw] (b) at (6,0) {};
\node[circle,fill=black,inner sep=0.8pt,draw] (c) at (3,4.2) {};
\node[circle,fill=black,inner sep=1.2pt,draw] (e) at (5.5,0) {};
\node[circle,fill=black,inner sep=1.2pt,draw] (f) at (3.3,3.78) {};
\node[circle,fill=black,inner sep=1.2pt,draw] (j) at (0.5, 0) {};
\node[circle,fill=black,inner sep=1.2pt,draw] (k) at (5,0) {};

\node () at (-0.7,0) {\tiny$(a,0,0)$};
\node () at (6.7,0) {\tiny$(0,\frac{d}{w_1},0)$};
\node () at (3.3,4.4) {\tiny$(0,0,\frac{d}{w_2})$};
\node () at (5.5,-0.4) {\tiny$(0,b,1)$};
\node () at (4.1,3.78) {\tiny$(1,0,c)$};

\draw (a)--(b)--(c)--(a);
\draw (e)--(a)--(f)--(e);
\draw [decorate] ([yshift=-1mm]a.west) --node[yshift=-5mm][yshift=1mm]{$k+1$} ([yshift=-1mm]e.east);
\end{tikzpicture}$$
 
 The triangle $\Delta$ has $k+2$ boundary lattice points and $g$ interior lattice points. By Lemma \ref{triangulation}, we can partition $\Delta$ into $2g+k$ triangles. Assume by contradiction at least one small tetrahedron has volume at least $d$, therefore $\det(\Delta)>(2g+k)d$. On the other hand:
 $$\det(\Delta)=kdc=kd\frac{2g+k}{k}=(2g+k)d,$$
 so we reached a contradiction.
 
 \item[(iii)] If the distinguished triangle is $\begin{pmatrix}
 a&&0&&0\\
 1&&b&&0\\
 1&&0&&c\end{pmatrix},$ there exists $k\in\Z$ such that:
 
 \begin{align*}
 &a=1+kw_1w_2\\
 &b=kw_0w_2\\
 &c=kw_0w_1\\
 &d=kw_0w_1w_2+w_0\\
 &\det(\Delta)=abc=kw_0w_1w_2d=d(d-w_0)
 \end{align*}
 
In this case, the genus formula gives us:
\begin{align*}
2gw_0w_1w_2=&d(d-\sum w_i)+w_0w_1+w_0w_2\\
2gw_0w_1w_2=&aw_0(d-\sum w_i)+w_0w_1+w_0w_2\\
2gw_1w_2=&(kw_1w_2+1)(kw_0w_1w_2-w_1-w_2)+w_1+w_2\\
2g=&k(kw_0w_1w_2+w_0-w_0-w_1)\\
2g=&k(d-w_1-w_2).
\end{align*}

$$
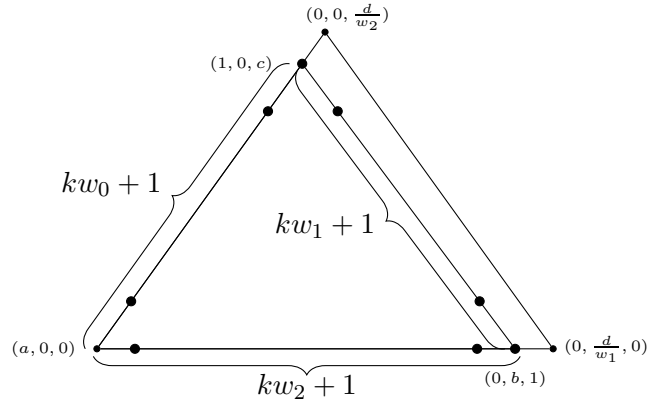
\begin{tikzpicture}[level/.style={},decoration={brace,mirror,amplitude=7}][H]
\node[circle,fill=black,inner sep=0.8pt,draw] (a) at (0,0) {};
\node[circle,fill=black,inner sep=0.8pt,draw] (b) at (6,0) {};
\node[circle,fill=black,inner sep=0.8pt,draw] (c) at (3,4.2) {};
\node[circle,fill=black,inner sep=1.2pt,draw] (e) at (5.5,0) {};
\node[circle,fill=black,inner sep=1.2pt,draw] (f) at (2.7,3.78) {};
\node[circle,fill=black,inner sep=1.2pt,draw] (g) at (0.5, 0) {};
\node[circle,fill=black,inner sep=1.2pt,draw] (h) at (5,0) {};
\node[circle,fill=black,inner sep=1.2pt,draw] (i) at (0.45,0.63) {};
\node[circle,fill=black,inner sep=1.2pt,draw] (j) at (2.25,3.15) {};
\node[circle,fill=black,inner sep=1.2pt,draw] (k) at (5.0333,0.63) {};
\node[circle,fill=black,inner sep=1.2pt,draw] (l) at (3.1666,3.15) {};

\node () at (-0.7,0) {\tiny$(a,0,0)$};
\node () at (6.7,0) {\tiny$(0,\frac{d}{w_1},0)$};
\node () at (3.3,4.4) {\tiny$(0,0,\frac{d}{w_2})$};
\node () at (5.5,-0.4) {\tiny$(0,b,1)$};
\node () at (1.9,3.78) {\tiny$(1,0,c)$};

\draw (a)--(b)--(c)--(a);
\draw (e)--(a)--(f)--(e);
\draw [decorate] ([yshift=-1mm]a.west) --node[yshift=-4mm]{$kw_2+1$} ([yshift=-1mm]e.east);
\draw [decorate] ([xshift=-1mm]f.west) --node[xshift=-10mm][yshift=3mm]{$kw_0+1$}([xshift=-2mm]a.east);
\draw [decorate] ([yshift=-1mm]f.west) --node[xshift=-10mm][yshift=-2mm]{$kw_1+1$}([xshift=-2mm]e.east);
\end{tikzpicture}$$

There are $k(w_0+w_1+w_2)$ lattice points on the boundary of $\Delta$ and $g-kw_0+1$ interior lattice points, so we can triangulate $\Delta$ into $2(g-kw_0+1)+k(w_0+w_1+w_2)-2=2g-kw_0+kw_1+kw_2=k(d-w_0)$ smaller triangles.

If one small tetrahedron has volume at least $d$, then 
$$dk(d-w_0)<d(d-w_0),$$
which is clearly a contradiction.

 \item[(c)] Two weights divide $d$. Up to a permutation of the weights, the distinguished triangle $\Delta$ is 
 $$\begin{pmatrix}
 a&&0&&0\\
 0&&b&&0\\
 1&&0&&c\end{pmatrix}.$$
 In this case there exist $k,l\in\Z$ such that:
 \begin{align*}
 &a=kw_1=lw_2+1\\
 &b=kw_0\\
 &c=lw_0\\
 &d=kw_0w_1=lw_0w_2+w_0\\
 &\det(\Delta)=abc=kdc=bld
 \end{align*}
 
 The genus formula (\ref{genus-formula}) gives:
 \begin{align*}
 2g=&\frac{d(d-\sum w_i)}{w_0w_1w_2}+\frac{1}{w_2}+1\\
 (2g-1)w_0w_1w_2=&d(d-\sum w_i)+w_0w_1\\
 (2g-1)w_0w_1w_2=&kw_0w_1(d-\sum w_i)+w_0w_1\\
 (2g-1)w_2=&k(d-\sum w_i)+1\\
 (2g-1)w_2=&k(lw_0w_2-w_1-w_2)+1\\
 (2g-1)w_2=&klw_0w_2-kw_1-kw_2+1\\
 (2g-1)w_2=&klw_0w_2-lw_2-kw_2\\
 2g-1=&klw_0-l-k\\
 2g+l+k-1=&klw_0.
 \end{align*}

$$
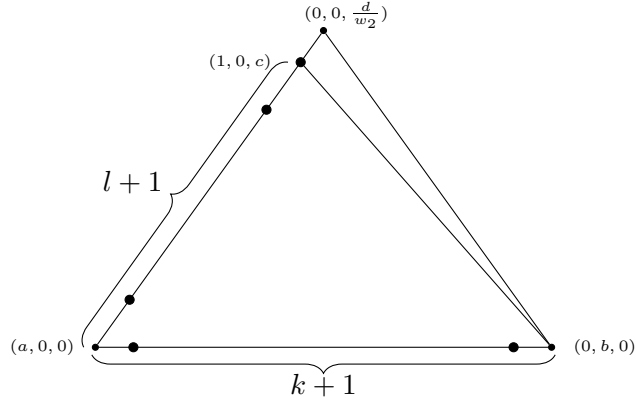
\begin{tikzpicture}[level/.style={},decoration={brace,mirror,amplitude=7}][H]
\node[circle,fill=black,inner sep=0.8pt,draw] (a) at (0,0) {};
\node[circle,fill=black,inner sep=0.8pt,draw] (b) at (6,0) {};
\node[circle,fill=black,inner sep=0.8pt,draw] (c) at (3,4.2) {};
\node[circle,fill=black,inner sep=1.2pt,draw] (d) at (2.7,3.78) {};

\node[circle,fill=black,inner sep=1.2pt,draw] (g) at (0.5, 0) {};
\node[circle,fill=black,inner sep=1.2pt,draw] (h) at (5.5,0) {};
\node[circle,fill=black,inner sep=1.2pt,draw] (i) at (0.45,0.63) {};
\node[circle,fill=black,inner sep=1.2pt,draw] (j) at (2.25,3.15) {};

\node () at (-0.7,0) {\tiny$(a,0,0)$};
\node () at (6.7,0) {\tiny$(0,b,0)$};
\node () at (3.3,4.4) {\tiny$(0,0,\frac{d}{w_2})$};
\node () at (1.9,3.78) {\tiny$(1,0,c)$};

\draw (a)--(b)--(c)--(a);
\draw (d)--(b);
\draw [decorate] ([yshift=-1mm]a.west) --node[yshift=-4mm]{$k+1$} ([yshift=-1mm]b.east);
\draw [decorate] ([xshift=-1mm]f.west) --node[xshift=-7mm][yshift=3mm]{$l+1$}([xshift=-2mm]a.east);

\end{tikzpicture}$$

 This distinguished triangle has $k+l+1$ boundary lattice points and $g$ interior lattice points, so we can triangulate it into $2g+l+k-1$ small triangles. If one small tetrahedron has volume at least $d$, then we should have
 \begin{align*}
 d(2g+l+k-1)< &\det(\Delta)=kdc\\
 2(2g+l+k-1)<& kc\\
 2klw_0<& klw_0,
 \end{align*}
 which is again a contradiction.
 
 \item[(d)] All weights divide $d$. In this case, the distinguished minor $\Delta$ is $\begin{pmatrix}
 a&0&0\\
 0&b&0\\
 0&0&c\end{pmatrix}.$
 
There exists an integer $k>0$ such that:
\begin{align*}
&a=kw_1w_2\\
&b=kw_0w_2\\
&c=kw_0w_1\\
&d=kw_0w_1w_2\\
&\det(\Delta)=k^3w_0^2w_1^2w_2^2=kd^2
\end{align*}
The genus formula (\ref{genus-formula}) gives us the following:
\begin{align*}
2g-2=&\frac{d(d-\sum w_i)}{w_0w_1w_2}\\
(2g-2)w_0w_1w_2=&kw_0w_1w_2(d-\sum w_i)\\
2g-2+\sum kw_i=&kd.
 \end{align*}
 
 $$
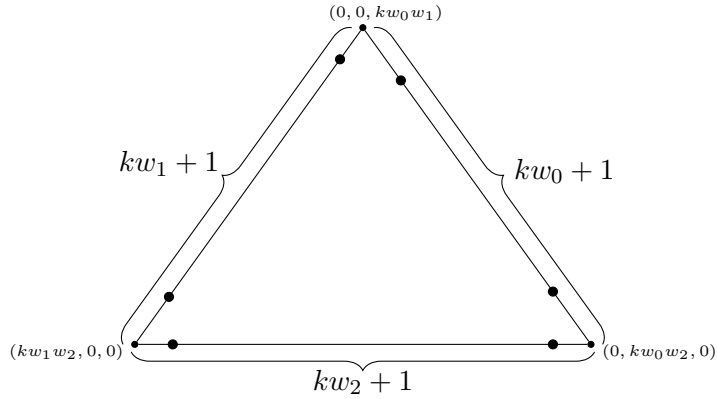
\begin{tikzpicture}[level/.style={},decoration={brace,mirror,amplitude=7}][H]
\node[circle,fill=black,inner sep=0.8pt,draw] (a) at (0,0) {};
\node[circle,fill=black,inner sep=0.8pt,draw] (b) at (6,0) {};
\node[circle,fill=black,inner sep=0.8pt,draw] (c) at (3,4.2) {};
\node[circle,fill=black,inner sep=1.2pt,draw] (d) at (2.7,3.78) {};
\node[circle,fill=black,inner sep=1.2pt,draw] (g) at (0.5, 0) {};
\node[circle,fill=black,inner sep=1.2pt,draw] (h) at (5.5,0) {};
\node[circle,fill=black,inner sep=1.2pt,draw] (i) at (0.45,0.63) {};
\node[circle,fill=black,inner sep=1.2pt,draw] (j) at (5.5,0.7) {};
\node[circle,fill=black,inner sep=1.2pt,draw] (k) at (3.5,3.5) {};

\node () at (-0.9,-0.1) {\tiny$(kw_1w_2,0,0)$};
\node () at (6.9,-0.1) {\tiny$(0,kw_0w_2,0)$};
\node () at (3.3,4.4) {\tiny$(0,0,kw_0w_1)$};

\draw (a)--(b)--(c)--(a);

\draw [decorate] ([yshift=-1mm]a.west) --node[yshift=-4mm]{$kw_2+1$} ([yshift=-1mm]b.east);
\draw [decorate] ([xshift=-1mm]c.west) --node[xshift=-9mm][yshift=3mm]{$kw_1+1$}([xshift=-2mm]a.east);
\draw [decorate] ([xshift=2mm]b.west) --node[xshift=10mm][yshift=2mm]{$kw_0+1$}([xshift=1mm]c.east);
\end{tikzpicture}$$
 
 The distinguished triangle has $kw_0+kw_1+kw_2$ points on the boundary and $g$ interior points, hence by Lemma \ref{triangulation}, we can triangulate $\Delta$ into $2g+kw_0+kw_1+kw_2-2=kd$ smaller disjoint triangles. Each of these triangles together with the origin forms a small tetrahedrons of volume equal to a positive multiple of $\frac{d}{2}$. If any of these volumes is at least $d$, then
 $$\det(\Delta)>d\cdot kd=kd^2,$$
 which is a contraction, so we are done.
 \end{enumerate}
 \end{proof}

 \begin{lemma}
 Let $Q=(w_0,w_1,w_2,d)$ be a $g$-good quadruple, and let $M(P)$ the corresponding $n(P)\times 3$ matrix. Let $v_1,v_2,v_3$ be three lattice points in $P$ that form a minor whose determinant is $\pm d$. Then any row vector in $M(P)$ is an integer combination of $v_1,v_2,v_3$.
 \end{lemma}
 
 \begin{proof}
 Let $v$ a row in $M(P)$. There exist unique $\alpha_i\in\R$ such that $v=\alpha_1v_1+\alpha_2v_2+\alpha_3v_3$. By symmetry, it suffices to show $a_3\in\Z$.
 We know
 $$\det\begin{pmatrix}
 -&&v_1&&-\\
 -&&v_2&&-\\
 -&&v&&-\end{pmatrix}=\pm\alpha_3d$$
is a multiple of $d$, by Lemma \ref{all-minors-div-d}, so $\alpha_3\in \Z$.
 \end{proof}

 As a consequence of the above lemma, we can associate to each $g$-good quadruple $Q=(w_0,w_1,w_2,d)$ a new polytope in $\Z^2$ as follows. Let $M(P)$ be the $n(P)\times 3$ matrix associated to the quadruple, and let $v_1,v_2,v_3$ be three row vectors in $M(P)$ whose associated determinant is $\pm d$. Consider the transformation to $\Z^2$:
 \begin{align*}
 \label{map-to-Z2}
v_1&\to e_1\\
v_2&\to e_2\\
v_3&\to 0.
\end{align*}
Since all the other rows in $M(P)$ are integer combinations of the $v_i$'s, the lattice points of $P$ are sent to lattice points in $\Z^2$. We call the new polytope $\overline{P}$.
The number $n=n(P)=n(\overline{P})$ is invariant under this transformation. Recall the number of interior points is always $g$.

 \begin{lemma}
 \label{unif-bound}
Fix $g\geq 1$. As $Q=(w_0,w_1,w_2,d)$ varies over all $g$-good quadruples, there exists a uniform bound on $n=n(P)$ as a function of $g$, namely $n\leq 3g+6$.
\end{lemma}

\begin{proof}
Let $(w_0,w_1,w_2,d)$ be a $g$-good quadruple, and let $\overline{P}$ be its associated polytope in $\Z^2$. $\overline{P}$ has exactly $g$ interior lattice points, where $g\geq 1$, therefore $n(P)\leq 3g+6$. \cite[4.5.6]{koelman}


\end{proof}

We say two convex polytopes $P_1,P_2\subset\Z^2$ are equivalent if there is an affine transformation $f:\Z^2\to \Z^2$ such that $f(P_1)=P_2$.

\begin{lemma}\cite[4.4]{koelman}\label{finite-polygons}
For a fixed $n\in\Z_+$, there are only finitely many classes of equivalent polytopes in $\Z^2$ with $n(P)=n$.
\end{lemma}

\begin{proof}

Let $P$ be a polynomial in $\Z^2$ containing $n$ lattice points. If $g(P)=0$, then $P$ is equivalent to one of following: a triangle $I_{n-2}$ with vertices $(0,0),(n-2,0),(0,1)$, a polytope $I_{k,l}$ with vertices $(0,0),(k,0),(l,1),(0,1)$, where $k\geq l\geq 1$ and $k+l=n-2$, or the triangle $I$ with vertices $(0,0),(2,0),(0,2)$ (only in the special case when $n=6$).

If $g\geq 1$, in particular $n\geq 4$. Pick $Q$ any vertex of $P$ and let $P'$ be the polytope obtained by deleting $Q$ from $P$. Then $n(P')=n-1$. By induction, if we know there are finitely many $P'$ with $n(P')=n-1$, then for all such $P'$ we determine the points $Q$ which are at height 1 with respect to some vertex, thus obtaining a finite number of equivalent classes of polytopes with $n(P)=n$.
\end{proof}

We are now ready to prove the statement of the main theorem, and show there are only finitely many closed loci  of weighted plane curves $Z_Q\subset M_g$, as $Q$ varies over all $g$-good quadruples $(w_0,w_1,w_2,d)$.

\begin{proof}[Proof of Main Theorem]
Fix a genus $g\geq 1$. By Lemma \ref{unif-bound}, every $g$-good quadruple has an associated polytope in $\Z^2$ with at most $3g+6$ lattice points. By Lemma \ref{finite-polygons}, we know there are only finitely many classes of such polytopes, up to the affine equivalence defined above. We will show below that two equivalent polytopes in $\Z^2$ correspond to the same locus $Z_Q$ in the moduli space $M_g$.

Let $Q=(w_0,w_1,w_2,d)$ and $Q'=(w'_0,w'_1,w'_2,d')$ be two $g$-good quadruples whose corresponding polytopes $\overline{P}$ and $\overline{P'}$ in $\Z^2$ are equivalent, i.e. there exists $A\in SL_2(\Z)$ such that the function $f:\Z^2\to\Z^2$ which maps $(a,b)\to (a,b)\cdot A$ satisfies $f(\overline{P})=\overline{P'}$.

The existence of $f$ implies that $n(\overline{P})\leq n(\overline{P'})$; since $f^{-1}$ corresponds to the matrix $A^{-1}\in SL_2(\Z)$, then $n(\overline{P})=n(\overline{P'})=n$. Thus the polytopes $P$ and $P'$ have associated matrices $M(P)$ and $M(P')$ of the same dimension. We claim that $M(P)$ and $M(P')$ are equivalent up to a linear transformation, i.e., there exists $T\in GL_3(\Q)$ such that
$$M(P)\cdot T=M(P').$$
To see this, recall that we have a map $i_P:P\to\overline{P}$ which takes every vector
$$v=\alpha_1v_1+\alpha_2v_2+\alpha_3v_3\xrightarrow{i_P} (\alpha_1,\alpha_2),$$
whose inverse is given by
$$i_P^{-1}:\overline{P}\to P,\text{    }(\alpha_1,\alpha_2)\to \alpha_1v_1+\alpha_2v_2+(1-\alpha_1-\alpha_2)v_3.$$



Now, $T$ has entries rational numbers whose denominators are all divisors of $d$. To see this, let $A$ be a $3\times 3$ minor in $M(P)$ with determinant $\pm d$, and let $B=f(A)=A\cdot T$. The transformation $T$ must map $A$ to a minor with the smallest nonzero determinant in $M(P')$, so $\det(B)=d'$. Now $T=BA^{-1}$; $A^{-1}$ has rational entries whose denominators are all divisors of $\det(A)=\pm d$, and $B$ has integer entries, so the conclusion follows.

The matrix $T$ induces a mapping
\begin{align*}
x_0\to &(y_0)^{t_{11}}(y_1)^{t_{12}}(y_2)^{t_{13}}\\
x_1\to&(y_0)^{t_{21}}(y_1)^{t_{22}}(y_2)^{t_{23}}\\
x_2\to &(y_0)^{t_{31}}(y_1)^{t_{32}}(y_2)^{t_{33}},
\end{align*}
which is not a ring map. To fix this, we look instead at the graded subalgebra $\C[x_0,x_1,x_2]^{(d)}$ generated by monomials of degree divisible by $d$. Similarly, we consider the subalgebra $\C[y_0,y_1,y_2]^{(d')}$ generated by monomials of degree divisible by $d'$. The generators in degree 1 of $\C[x_0,x_1,x_2]^{(d)}$ are monomials $x_0^ax_1^bx_2^c$ such that $aw_0+bw_1+cw_2=d$; these monomials correspond to the row vectors in $M(P)$. Similarly, the generators in degree 1 of $\C[y_0,y_1,y_2]^{(d')}$ correspond to the row vectors in $M(P')$. Let $S_1$ be the subalgebra of $\C[x_0,x_1,x_2]^{(d)}$ generated by the monomials of degree 1, and $S'_1$ be the subalgebra of $\C[y_0,y_1,y_2]^{(d')}$ generated by the monomials of degree 1. In terms of varieties, we get the following diagram:

$$\begin{tikzcd}
\Proj(S_{1}) \arrow{d}{\cong} & (\PP^2(W))^{(d)}=\Proj(\C[x_0,x_1,x_2]^{(d)})\arrow[dashed]{l}{bir} &\PP^2(W)=\Proj (\C[x_0,x_1,x_2])\arrow{l}{\cong}\arrow[dashed]{d}{bir} \\
\Proj(S'_{1}) & (\PP^2(W'))^{(d')}=\Proj(\C[y_0,y_1,y_2]^{(d')})\arrow[dashed]{l}{bir} & \PP^2(W')=\Proj (\C[y_0,y_1,y_2])\arrow{l}{\cong}
\end{tikzcd}$$

The isomorphisms $(\PP^2(W))^{(d)}\cong \PP^2(W)$ and $(\PP^2(W'))^{(d')}\cong \PP^2(W')$ are elementary. The left isomorphism $\Proj(S_{1})\cong\Proj(S'_{1})$ is given by $T$: on the level of rings, every generator $x_0^ax_1^bx_2^c$ of $S_{1}$ is mapped to $y_0^{a'}y_1^{b'}y_2^{c'}$, where $(a,b,c)\cdot T=(a', b', c')$. Since $M(P)\cdot T=M(P')$, this mapping gives an isomorphism of projective varieties.

The inclusion of rings $S_1\hookrightarrow \C[x_0,x_1,x_2]^{(d)}$ induces a rational map 
$$\Proj S_1\dashrightarrow \Proj (\C[x_0,x_1,x_2])^{(d)}.$$

We claim this map is a birational morphism, defined on the dense open set $U=\{x_i\neq 0\}_{i=0,1,2}$. We prove this by showing that the fraction fields of $\C[x_0,x_1,x_2]^{(d)}$ and $S_1$ are the same.


Consider a monomial $x_0^ax_1^bx_2^c$ of degree $kd$, where $k\geq 1$. Let $v_1,v_2,v_3$ be again three vectors in $M(P)$ with corresponding minor $\pm d$. There exist unique $\alpha_i\in\R$ such that $(a,b,c)=\alpha_1v_1+\alpha_2v_2+\alpha_3v_3$. We claim that $\alpha_i\in\Z$. By symmetry, it suffices to show $\alpha_3\in \Z$.
We have 
$$\det\begin{pmatrix}
-&&v_1&&-\\
-&&v_2&&-\\
a&&b&&c\end{pmatrix}=\alpha_3d.$$
Additionally, since $aw_0+bw_1+cw_2=kd$, we can use linear transformations to obtain
$$w_0(\alpha_3d)= w_0\cdot \det\begin{pmatrix}
-&&v_1&&-\\
-&&v_2&&-\\
a&&b&&c\end{pmatrix}=\dots
=\det\begin{pmatrix}
d&&*&&*\\
d&&*&&*\\
kd&&b&&c\end{pmatrix}\in d\Z,$$
so $\alpha_3\in\frac{1}{w_0}\Z$. Similarly, we get that $\alpha_3\in\frac{1}{w_1}\Z$. Since $\gcd(w_0,w_1)=1$, it follows that $\alpha_3\in\Z$, as claimed.\\

Therefore, the monomial $x_0^ax_1^bx_2^c$ in $\C[x_0,x_1,x_2]_k^{(d)}$ can be expressed as a rational function of the three degree-$d$ monomials corresponding to $v_0,v_1,v_2$. This means that the function fields of $\C[x_0,x_1,x_2]^{(d)}$ and $S_1$ are isomorphic. This gives the desired birational map $\PP^2(W)^{(d)}\dasharrow \Proj(S_1)$, defined on the open set $\{x_1\neq 0\}_{i=0,1,2}$. We obtain the second birational morphism $\PP(W')^{(d)}\dasharrow \Proj(S^{\prime}_1)$ in the same fashion.\\




Now that we have defined a specific birational morphism from $\PP(^2W)\dasharrow\PP^2(W')$, defined on the open sets $U=\{x_i\neq 0\}_{i=0,1,2}\subset \PP^2(W)$ and $U'=\{y_i\neq 0\}_{i=0,1,2}\subset \PP^2(W')$, respectively, we want to see what are the images of smooth curves of genus $g$ under this rational map. Let $C_d=V(f)$ be a degree $d$-curve on $\PP^2(W)$, where $f(\underline{x})=\sum_i a_i\underline{x}^{v_i}$. On the level of rings, $f(\underline{x})$ is mapped to $f'(\underline{y})=\sum_i a_i\underline{y}^{v'_i}$, where $v_i\cdot T=v'_i$. Therefore, the closure of the image of $C_d=V(f)\subset \PP(W)$ is a new curve $C_{d'}=V(f')\subset \PP^2(W')$. 

Let $\PP V_Q$ parametrize the degree-$d$ curves on $\PP^2(W)$, and $\PP V_{Q'}$ parametrize the degree-$d'$ curves on $\PP^2(W')$. We have a map $\PP V_Q\xrightarrow{\cong}\PP V_{Q'}$ that sends $[C_d=V(f)]\to [C_{d'}=V(f')]$. Let $\mathcal{U}^{sm}_{Q}\subset \PP V_Q$ be the (nonempty) open dense set of smooth curves on $\PP^2(W)$ of degree $d$ (and genus $g$), and let $\mathcal{U}^{sm}_{Q'}\subset \PP V_{Q'}$ be the (nonempty) open dense set of smooth curves of degree $d$ (and genus $g$) on $\PP^2(W')$. The birational map $\PP^2(W)\dasharrow \PP^2(W')$ induces a birational morphism:

\[
\begin{tikzcd}
\mathcal{U}^{sm}_{Q} \arrow[dashed]{rr}{bir}\arrow{dr}{i}  && \mathcal{U}^{sm}_{Q'}\arrow{dl}{i'}\\
{} & M_g
\end{tikzcd}
\]
It follows immediately that $Z_Q=\overline{i(\mathcal{U}^{sm}_{Q})}=\overline{i'(\mathcal{U}^{sm}_{Q'})}=Z_{Q'}$, so the Main Theorem holds.

\end{proof}


\newpage
 
\begin{thebibliography}{9}

\bibitem{dolgachev}
Igor Dolgachev.
\textit{Weighted projective spaces}.
Lecture Notes in Mathematics, pages 34-71, 1982.

\bibitem{fletcher}
A. R. Fletcher.
\textit{Working with weighted complete intersections}.
Max-Plank-Institut fur Mathematik, 1989.

\bibitem{koelman}
R. J. Koelman.
\textit{The number of moduli of families of curves on toric varieties}.
Doctoral Thesis, Radboud University Nijmegen, 1991.

\end{thebibliography}
\end{document}